\documentclass[reqno]{amsart}

\pdfoutput=1

\usepackage[utf8]{inputenc}
\usepackage[english]{babel}
\usepackage{amsmath}
\usepackage{amsthm}
\usepackage{amsfonts}
\usepackage{amssymb}
\usepackage[autostyle]{csquotes} 
\usepackage[backend=bibtex, maxnames=4, doi=false, url=false, style=alphabetic]{biblatex}
\usepackage{caption}
\usepackage{hyperref} 
\usepackage[english]{cleveref} 
\usepackage{diagbox}
\usepackage{dsfont}
\usepackage{emptypage}
\usepackage[shortlabels]{enumitem}
\usepackage{extpfeil}
\usepackage{etoolbox}
\usepackage{faktor}
\usepackage{fancyhdr}
\usepackage{float}
\usepackage{graphicx}
\usepackage{indentfirst}
\usepackage{mathdots}
\usepackage{mathtools}
\usepackage{mathrsfs}
\usepackage{multirow}
\usepackage{rotating}
\usepackage{stackrel}
\usepackage{stmaryrd}
\usepackage{subcaption}
\usepackage{tikz}
\usepackage{verbatim}
\usepackage{young}

\usetikzlibrary{patterns,arrows,calc,matrix,decorations.pathreplacing}

\newcommand*\Z{\mathbb{Z}}
\newcommand*\N{\mathbb{N}}
\newcommand*\Q{\mathbb{Q}}
\newcommand*\C{\mathbb{C}}
\newcommand*\A{\mathcal{A}}

\newcommand*\R{\mathcal{R}}
\newcommand*\V{\mathbb{V}}
\renewcommand*\S{\mathscr{S}}
\newcommand{\stirling}[2]{\genfrac{[}{]}{0pt}{}{#1}{#2}}
\newcommand{\defeq}{\stackrel{\textnormal{def}}{=}}

\newtheorem{thm}{Theorem}[section]

\newtheorem{prop}[thm]{Proposition}

\newtheorem{lemma}[thm]{Lemma}
\theoremstyle{definition}
\newtheorem{defi}[thm]{Definition}
\newtheorem{example}[thm]{Example}
\newtheorem{question}[thm]{Question}
\theoremstyle{remark}
\newtheorem{remark}[thm]{Remark}

\DeclareMathOperator{\codim}{codim}

\DeclareMathOperator{\dd}{d}

\DeclareMathOperator{\gr}{gr}

\DeclareMathOperator{\im}{Im}
\DeclareMathOperator{\Ind}{Ind}

\DeclareMathOperator{\pr}{pr}

\DeclareMathOperator{\Res}{Res}
\DeclareMathOperator{\rk}{rk}
\DeclareMathOperator{\SL}{SL}

\DeclareMathOperator{\shape}{Shape}
\DeclareMathOperator{\supp}{Supp}
\DeclareMathOperator{\sgn}{sgn}

\DeclareMathOperator{\T}{T}

\makeatletter
\newcommand*{\bigcdot}{%
  {\mathbin{\mathpalette\bigcdot@{}}}%
}
\newcommand*{\bigcdot@scalefactor}{.75}
\newcommand*{\bigcdot@widthfactor}{1.4}
\newcommand*{\bigcdot@}[2]{%
  \sbox0{$#1\vcenter{}$}
  \sbox2{$#1\cdot\m@th$}%
  \hbox to \bigcdot@widthfactor\wd2{%
    \hfil
    \raise\ht0\hbox{%
      \scalebox{\bigcdot@scalefactor}{%
        \lower\ht0\hbox{$#1\bullet\m@th$}%
      }%
    }%
    \hfil
  }%
}
\makeatother

\robustify{\bigcdot}

\makeatletter
\g@addto@macro{\UrlBreaks}{\UrlOrds}
\g@addto@macro{\UrlBreaks}{%
\do\/\do\d%
}
\makeatother

\begin{document}
\title[Combinatorics and cohomology of elliptic arrangements]{The connection between combinatorics and cohomology of elliptic arrangements}
\author[R. Pagaria]{Roberto Pagaria}
\address{Roberto Pagaria}
\address{Scuola Normale Superiore\\ Piazza dei Cavalieri 7, 56126 Pisa\\ Italia}\email{roberto.pagaria@sns.it}

\begin{abstract}
We prove that the cohomology algebra of elliptic arrangements depends only on the poset of layers.
In the particular case of braid elliptic arrangements, we study the cohomology as representation and we compute some Hodge numbers.
Finally, we discuss $1$-formality for graphic elliptic arrangements.
\end{abstract}
\maketitle
\vspace{-1cm}
\section*{Introduction}

Elliptic arrangements are a generalizations of hyperplane arrangements and a particular case of Lie group arrangements (definitions in \cite{OT92, LTY17}).
The Betti numbers of Lie group arrangements are known in the case of non-compact group and this numbers depend only on the associated matroid (see \cite{LTY17}).
The smallest non-trivial cases of compact Lie group arrangements are elliptic arrangements.

Little is known about elliptic arrangements: a model for the cohomology, a formula for the Euler characteristic (\cite{Bibby16}) and the cohomology with coefficients in a generic local system (\cite{LV12}).
We have investigated the dependency of mixed Hodge numbers from the associated matroid and computed some of this numbers for the braid elliptic arrangements.

An example of two central elliptic arrangements with the same Tutte polynomial but with different Poincar\'e polynomials is show in \cite{Pag18}.

The main tool -- used for computing the cohomology -- is a model that can be obtained as the second page of the Leray spectral sequence \cite{Bibby16, Dupont15} or as a specialization of the  K\v{r}\'\i\v{z} model \cite{Kriz94, Totaro96, Bezrukavnikov, AAB14, Azam15}.

We start the discussion in \cref{sec:ellipt_arr} giving a combinatorial model for the rational cohomology of the complement of any elliptic arrangement.
This model depends only on the poset of layers of the elliptic arrangement.
In \cref{sec:braid_arr} we specialize the model to the case of braid elliptic arrangements and we find some linearly independent cohomology classes using the bi-grading given by the Mixed Hodge Structure.
Then we study the representation theory of the combinatorial model and of the cohomology both as representations of $\mathfrak{S}_n$ and of $\operatorname{SL}_2(\mathbb{Q})$ (see \cref{sec:rep_theory}).
Finally, we discuss in \cref{sec:formality} the formality in the setting of graphic elliptic arrangements, a class of arrangements including the braid elliptic arrangements.


\section{Elliptic arrangements}\label{sec:ellipt_arr}
\subsection{Definitions}\label{subsect:def}
Let $E$ be an elliptic curve and consider, in the Cartesian product, $E^n$ the divisor $D$ defined by the equation 
\[\sum_{i=0}^n m_i P_i = Q\]
which depends on parameters $m_i\in \Z$, $i=1,\dots, n$ and a point $Q \in E$.
We denote by $\tilde{D}=\{(P_1, \dots, P_n) \in E^n \mid \sum_{i=0}^n m_i P_i = O\}$ the translated of $D$ at the origin $O$ of the elliptic curve.
The divisor $D$ is connected if and only if $\gcd (m_i | i=1,\dots, n)=1$;from now on 
we will consider only connected divisors.
An \textit{elliptic arrangement} $\A=\{D_1, \dots, D_l\}$ is a finite collection of such divisors in $E^n$.
Notice that the equations $\sum_{i=0}^n m_i P_i = Q$ and $\sum_{i=0}^n (-m_i) P_i = -Q$ define the same divisor;
a \textit{defining matrix} $N$ in $\operatorname{M}(n,l;\Z)$ for an elliptic arrangement $\A$ is a matrix whose $i^{\textnormal{th}}$ column defines the divisor $\tilde{D}_i$.

We denote by $M(\A)\subset E^n$ the \textit{complement} of these divisors:
\[ M(\A) \stackrel{\textnormal{def}}{=} E^n \setminus \bigcup_{D \in \A} D .\] 
We call \textit{layer} a connected component of the intersection of some divisors that belong to the arrangement $\A$.
If all possible non empty intersections of divisors in $\A$ are connected, we say that $\A$ is \textit{unimodular}.
The intersection of some divisors locally looks like the intersection of hyperplanes in $\C^n$: let $W$ be a layer and $p\in W $ a point, than the hyperplanes $\T_p D \subset \T_p E^n \simeq \C^n$ form an hyperplane arrangement 
\[\A_W=\{T_p D \mid D \supset W\}\]
that, up to translations does not depend on the point $p \in W$.

Define the \textit{poset of layers} as the partially ordered set $\S=\S(\A)$ whose elements are the layers of $\A$ ordered by reverse inclusions.
The poset of layers is a graded poset with rank function $\rk: \S \rightarrow \N$ given by codimension in $E^n$
\[ \rk (W) \stackrel{\textnormal{def}}{=} \codim_{E^n} (W).\]

This setting is analogous to the setting of the toric case \cite{DeCP05}, with the difference that the equation
\[nP=Q\]
for fixed $Q\in E$ and $n$, has $n^2$ solutions instead of $n$ solutions in the toric case.

\subsection{A model for the cohomology algebra} \label{subsect:model}
We briefly recall the result of \cite{Bibby16} and of \cite{Dupont15}.

Consider the bi-graded algebra $E_2^{\bigcdot,\bigcdot}(M(\A))$ over $\Q$ defined by
\[ E_2^{\bigcdot,\bigcdot}(M(\A)) \stackrel{\textnormal{def}}{=} \bigoplus_{W \in \S} H^\bigcdot(W,\Q)(-\rk W) \otimes_\Q H^{\rk W} (M(\A_W);\Q)\]
where $M(A_W)= \C^n \setminus \bigcup_{H \in \A_W} H$ is the complement of the hyperplane arrangement $\A_W$ and $(-\rk W)$ is the Tate twist of the mixed Hodge structure on the cohomology.
The Tate twist and the ring of coefficients of the cohomology are understood.
The homogeneous component of degree $(p,q)$ is
\[E_2^{p,q}(M(\A))= \bigoplus_{\rk W =q} H^p(W,\Q) \otimes_\Q H^{q} (M(\A_W);\Q).\]
Consider the layers $W_1, W_2$ and $L$ such that $L\geq W_1$, $L\geq W_2$ and $\rk L= \rk W_1 + \rk W_2$, let $\eta_i: M(\A_{L}) \rightarrow M(\A_{W_i})$ and $\gamma_i: L \rightarrow W_i$ be the natural inclusions.
Define the product of the two elements $x_i \otimes \omega_i$ in the homogeneous components $H^{p_i}(W_i,\Q) \otimes_\Q H^{\rk W_i} (M(\A_{W_i});\Q) $, $i=1,2$, as
\begin{equation}\label{eq:prod_astratto}
(x_1 \otimes \omega_1) \cdot (x_2 \otimes \omega_2)  \stackrel{\textnormal{def}}{=} (-1)^{p_2 \rk W_1}(\gamma_1^*x_1 \smile \gamma_2^*x_2) \otimes (\eta_1^* \omega_1 \smile \eta_2^* \omega_2)
\end{equation}
in the homogeneous addendum $H^{p_1+p_2}(L;\Q) \otimes H^{\rk L}(M(\A_L);\Q)$ for every connected components $L$ of $W_1 \cap W_2$.

Let $\sigma_L : H^{\rk W}(M(\A_W))\rightarrow H^{\rk L}(M(\A_L))$ be the composite function
\[H^{\rk W}(M(\A_W))\xrightarrow{\sigma} \bigoplus_{\substack{L'<W \\ \rk L'= \rk W -1}} H^{\rk L'}(M(\A_{L'})) \twoheadrightarrow H^{\rk L}(M(\A_L))\]
between the natural differential in cohomology of the hyperplane arrangement $\A_W$ defined in \cite[2.3]{OS80} and the canonical projection. 
\noindent
Let $G_L: H^p(W)(-\rk W) \rightarrow H^{p+2}(L)(-\rk L)$ be the Gysin morphism for the inclusion $W \hookrightarrow L$.
The differential $\dd$ of bi-degree $(2,-1)$ is defined on the homogeneous element $x\otimes \omega \in H^p(W)\otimes H^{\rk W}(M(\A_W))$ as 
\begin{equation} \label{eq:def_dd}
\dd (x\otimes \omega) \defeq \sum_{\substack{L<W \\ \rk L= \rk W -1}} G_L(x)\otimes \sigma_L(\omega).
\end{equation}

Notice that $E_2^{p,q}(M(\A))$ has weight $p+2q$ and that the differential graded algebra (dga) $(E_2^{\bigcdot,\bigcdot},\dd)$ is the second page of the Leray spectral sequence for the inclusion $M(\A) \hookrightarrow E^n$ of the constant sheaf $\underline{\Q}$.

The cohomology of the variety $M(\A)$ was studied by Christin Bibby and by Clémente Dupont.
The following result was proven, independently, in \cite[Theorem 3.3]{Bibby16} and in \cite[Theorem 1.2]{Dupont15}.
\begin{thm} \label{thm:Bibby-Dup}
The graded cohomology algebra $H^\bigcdot(M(\A);\Q)$ is isomorphic to the cohomology of the complex $(E_2^{\textnormal{tot}}(M(\A)),\dd)$ as graded algebra endowed with mixed Hodge structure over $\Q$.
\end{thm}

\subsection{The combinatorial model}
The model presented in this section coincides with the one presented in \cite[Theorem 1.2]{Dupont15}, specialized to the case of hypersurface arrangements in $E^n$.

Fix a total order on the divisors, numbering them such that $D_i < D_j$ if and only if $i<j$.
We say $I\subseteq [l] := \{1, \dots, l\}$ is an \textit{independent set} if the intersection $\bigcap_{i \in I} D_i$ has codimension equal to $|I|$ and we say that a layer $W\in \S$ and an independent set $I$ are \textit{associated} if $W$ is a connected component of $\bigcap_{i \in I} D_i$.
As usual, a subset $C\subset [l]$ is a \textit{circuit} if it is not independent and, for all $i\in C$, $C\setminus \{i\}$ is independent.

Let $F$ be the set
\[ F\stackrel{\textnormal{def}}{=}\{(W,I) \in \S \times \mathcal{P} ([l]) \mid W \textnormal{ and } I \textnormal{ are associated, } I \neq \emptyset\}.\]

Define the external algebra $\Lambda$ over 
$\Q$ with generators
\[\{x_i,y_i,\omega_{W,I}\}_{1 \leq i \leq l, \, (W,I) \in F}.\]
We set the degree of each $x_i$ and $y_i$ equal to $(1,0)$ and the degree of $\omega_{W,I}$ equal to $(0,|I|)$ (notice that $|I|=\rk W$).
Define the differential $\partial: \Lambda \rightarrow \Lambda$ of bi-degree $(2,-1)$ on generators as follows: for $i=1, \dots, l$ we impose $\partial x_i=0$ and $\partial y_i=0$ and for the other generators:
\[ \partial \omega_{W,I} \stackrel{\textnormal{def}}{=} \sum_{\substack{(L,J)\in F \\ J\subset I \, \, L<W \\ \rk L= \rk W -1}} (-1)^{\tau(j)} x_j y_j \omega_{L,J} \]
where $\{j\}= I \setminus J$ and $\tau(j)=|\{k \in J \mid k <j\}|$.

Let $N$ be a defining matrix for $\A$. 
We define the dga $A^{\bigcdot,\bigcdot}(\A,N)$ as the quotient of $\Lambda$ by the following relations:
\begin{enumerate}
\item \label{rel_1} For every $i \in I$
\begin{equation}\label{eq:x_i_omega_i}
x_i\omega_{W,I}=0 \quad \textnormal{and} \quad y_i\omega_{W,I}=0
\end{equation}

\item \label{rel_2} For each two pairs $(W_1,I_1)$ and $(W_2,I_2)$ in $F$, let $L$ be a connected components of $W_1 \bigcap W_2$.
If $(L,I_1 \cup I_2) \not \in F$, consider the relation
\begin{equation}\label{eq:prod}
\omega_{W_1,I_1}\omega_{W_2,I_2}=0
\end{equation}
otherwise the relation
\begin{equation}\tag{\ref{eq:prod}'}\label{eq:prod_2}
\omega_{W_1,I_1}\omega_{W_2,I_2} = (-1)^\sigma \sum_{L\in \pi_0 (W_1 \cap W_2)} \omega_{L,I_1 \cup I_2} 
\end{equation} 
where $\sigma$ is the sign of the permutation that reorders $I_1 \cup I_2$ and $L$ ranges over the connected components of $W_1 \bigcap W_2$.
\item \label{rel_3} For every circuit $C$ and for every connected component $W$ of $\bigcap_{i \in C} D_i$ consider the relation
\begin{equation}\label{eq:circuit}
\sum_{i \in C} (-1)^{C_{\leq i}} \omega_{W,C \setminus \{i\}}=0
\end{equation}
where $C_{\leq i} = | \{ c \in C \mid c<i\} | $.
\item \label{rel_4} The following hold:
\begin{equation}\label{eq:curva_ell}
\sum_{i=1}^l k_i x_i=0 \quad \textnormal{and} \quad \sum_{i=1}^l k_i y_i=0
\end{equation}
for every vector $\underline{k} \in \ker{N}\subset \Q^l$ .
\end{enumerate}

\begin{thm}\label{thm:comb_depend}
The structure of the rational cohomology ring of $M(\A) \subseteq E^n$ depends only on the graded poset $\S(\A)$ (and on $n$).
\end{thm}

An arrangement $\A$ is \textit{essential} if a defining matrix (or, equivalently, any defining matrix) has rank $n$.

\begin{lemma} \label{lemma:iso_dga}
Let $\A$ be an essential elliptic arrangement and $N$ be a defining matrix for $\A$.
The differential graded algebras $E_2^{\bigcdot,\bigcdot}(M(\A))$ and $A^{\bigcdot,\bigcdot}(\A,N)$ are isomorphic.
\end{lemma}

\begin{proof}
Fix a basis $\{x,y\}$ of $H^1(E;\Z)$ such that $xy$ generates $H^2(E;\Z)$ and let $N=(n_{i,j})$ be a defining matrix for the arrangement $\A$ (see definition in \ref{subsect:def}).
For each $i=1, \dots, l$, let $\pi_i: E^n \rightarrow E$ be the map
\[(P_1, \dots, P_n) \mapsto \sum_{j=1}^n n_{j,i} P_j\]
thus obviously $\tilde{D}_i=\pi_i^{-1}(O)$.
Consider the morphism $f:\Lambda \rightarrow E_2^{\bigcdot,\bigcdot}(M(\A))$ defined by
\begin{align*}
& f(x_i)= \pi_i^*(x)\otimes 1 & &\in H^1(E^n)\otimes H^0(M(\A_{E^n}))  \\
& f(y_i)= \pi_i^*(y)\otimes 1 & &\in H^1(E^n)\otimes H^0(M(\A_{E^n})) \\
& f(\omega_{W,I}) = 1 \otimes e_I & & \in H^0(W) \otimes H^{\rk W}(M(\A_W))
\end{align*}
where $e_I= e_{I_1}e_{I_2} \cdots e_{I_{|I|}}$ is the product of the generators $e_i \in H^1(M(\A_W))$ for the hyperplane $T_p D_i$ (for a point $p \in W$).
The map $f$ preserves the gradation and sends relations \eqref{eq:x_i_omega_i}-\eqref{eq:curva_ell} to zero.
Indeed:
\begin{enumerate}
\item If $i \in I$, then 
\begin{align*}
f(x_i \omega_{W,I}) &= f(x_i) \cdot f(\omega_{W,I})\\
&= (\pi_i^*(x) \otimes 1)\cdot (1\otimes e_I) \\
&= \gamma^* \pi_i^*(x) \otimes e_I \\
&= (\pi_i \circ \gamma)^*(x) \otimes e_I \in H^0(W) \otimes H^{\rk W}(M(\A_W))
\end{align*}
where $\gamma$ is the inclusion $W \hookrightarrow E^n$.
Since $W \subset D_i$, the composite function $\pi_i \circ \gamma$ is constant and so $f(x_i \omega_{W,I})=0$.

\item Relation \eqref{eq:prod_2} corresponds to the multiplication in $E_2^{\bigcdot,\bigcdot}(M(\A))$ given by equation \eqref{eq:prod_astratto}.
If $(L, I_1 \cup I_2)$ does not belong to $F$, then $\rk L< \rk W_1 + \rk W_2$.
In this case the arrangement $\A_L$ -- which has rank equal to $\rk L$ -- has no cohomology in degree $\rk W_1 + \rk W_2$, so the product $\eta_1^* e_{I_1} \smile \eta_2^* e_{I_2}$ in equation \eqref{eq:prod_astratto} is zero.
Hence $f$ maps \eqref{eq:prod} to zero.

\item The map $f$ sends relation \eqref{eq:circuit} to the Orlik-Solomon relation for the same circuit $C$ in the cohomology of the hyperplane arrangement $\A_W$.

\item Suppose that $\underline{k}=(k_1, \dots, k_l) \in \ker N$, then $N\underline{k}=0$ and so $\sum_{i=1}^l k_i \pi_i$ is the constant map. We have
\begin{align*}
f \left( \sum_{i=1}^l k_ix_i \right) &= \sum_{i=1}^l k_i f(x_i) = \sum_{i=1}^l k_i \pi_i^*(x) \otimes 1 \\
&= \left(\sum_{i=1}^l k_i \pi_i \right)^*(x) \otimes 1=0.
\end{align*}
The same proof holds for the variables $y_i$.
\end{enumerate}

Now we show the equality $f\circ \partial = \dd \circ f$ on the generators, where $\dd$ is the differential defined by equation \eqref{eq:def_dd}.
This is easy to check on the generators $x_i$ and $y_i$ since $f(\partial x_i)=0$ and $\dd (\pi_i^*x \otimes 1)=0$ (there is no layer of rank $-1$).
Then fix a generator $\omega_{W,I}$ and consider all pairs $(L,J)\in N$ such that $L<W$, $J \subset I$ and $\rk L= \rk W -1$.
Recall the Gysin map $G_L$ and the map in cohomology $\sigma_L$ introduced in \Cref{subsect:model}; the equality
\[ (-1)^{\tau(j)} \pi_j^* (xy)\otimes e_J = G_L(1) \otimes \sigma_L(e_I)\]
holds in $H^2(L)\otimes H^{\rk L}(M(\A_L))$ (where $\{j\}=I \setminus J$) since the equation of $W$ in $L$ is given by $\pi_j$ and $\sigma_L(e_I)=(-1)^{\tau(j)} e_J$.
The equality $f\circ \partial = \dd \circ f$ follows from
\begin{align*}
f(\partial \omega_{W,I}) &= f \left( \sum_{(L,J)} (-1)^{\tau(j)} x_jy_j \omega_{L,J} \right) \\
&= \sum_{(L,J)} (-1)^{\tau(j)} f(x_j) f(y_j) f(\omega_{L,J}) \\
&= \sum_{(L,J)} (-1)^{\tau(j)} (\pi_j^*(xy) \otimes e_J)_L \\
&= \sum_{(L,J)} G_L(1) \otimes \sigma_L(e_I)\\
&= \dd (1\otimes e_I) = \dd (f(\omega_{W,I}))
\end{align*}
where in the last equality we use the definition (eq. \eqref{eq:def_dd}) of the differential $\dd$.

Since the algebra $E_2^{\bigcdot,\bigcdot}(M(\A))$ is generated in degrees $(1,0)$ and $(0,q)$ for every $q>0$, in order to prove the surjectivity of $f$, we show that $E_2^{0,q}(M(\A))$ and $E_2^{1,0}(M(\A))$ are contained in $\im f$.
The inclusion $E_2^{0,q}(M(\A)) \subset \im f$ always holds since $H^0(W) \otimes H^{\rk W}(M(\A_W))$ is generated as vector space by $1 \otimes e_I$, where $I$ is an independent set of cardinality equal to $\rk W$.
The second inclusion $E_2^{1,0}(M(\A)) \subset \im f$ holds if and only if the arrangement $\A$ is essential, thus the surjectivity of $f$ follows by hypothesis.

Let $\varphi: A^{\bigcdot,\bigcdot}(\A,N) \rightarrow E_2^{\bigcdot,\bigcdot}(M(\A))$ be the map induced by $f$.
It is a well defined morphism of dga, preserves the bi-gradation and it is surjective by the previous discussion.
We prove injectivity using dimensional argument.
Consider the set $F' \subset F$ of pairs $(W,I)$ such that $I$ is a non-broken circuit for the arrangement $\A_W$ (see \cite[Definition 3.35 and Definition 3.36]{OT92}).
For each layer $W$ we choose a set $J(W)\subset [l]$ such that $\{ \pi_j^*(x), \pi_j^*(y)\}_{j\in J(W)}$ is a basis of $H^1(W;\Q)$, this is always possible since $\A$ is essential.
A basis of $A^{p,q}$ is given by $\prod_{k=1}^p z_k \omega_{W,I}$ where $(W,I)\in F'$ and $z_k$ are distinct elements of the set $\{ x_j, y_j\}_{j\in J(W)}$.
Thus the dimension of $A^{\bigcdot,\bigcdot}(\A,N)$ is
\[ \sum_{(W,I)\in F'}2^{\rk W}\]
and is equal to the dimension of $E_2^{\bigcdot,\bigcdot}(M(\A))$.
\end{proof}

Consider the group $G=\operatorname{GL}(n;\Q) \times (\Z/2\Z \wr \mathfrak{S} _l)$ (where $ \Z/2\Z \wr \mathfrak{S} _l$ is the wreath product of $\Z/2\Z$ by $\mathfrak{S}_l$) and its natural action on $\operatorname{M}(n,l;\Q)$: this action is given by the left multiplication of $\operatorname{GL}(n;\Q)$, by the sign reverse of the columns for $(\Z/2\Z)^l$ and by permuting the columns for $\mathfrak{S}_l$.
%

\begin{lemma}\label{lemma:orbit_space}
For all graded posets $\S$ the set
\[\{ N \in \operatorname{M}(n,l;\Z) \mid \exists \, \A \textnormal{ s.t. } \S \simeq \S(\A) \textnormal{ and } N \textnormal{ defines } \A \}\]
is contained in a unique $G$-orbit of $\operatorname{M}(n,l;\Q)$.
\end{lemma}

\Cref{lemma:orbit_space} follows from representations theory of arithmetic matroids (see, \cite[Corollary 3.8 and 3.9]{Pagaria17}).

\begin{proof}[Proof of \Cref{thm:comb_depend}]
Suppose $\A$ is an essential elliptic arrangement. 
By \Cref{thm:Bibby-Dup} and \Cref{lemma:iso_dga}, we need to show that the dga $A^{\bigcdot,\bigcdot}(\A,N)$ depends only on $\S(\A)$.
The construction of $A^{\bigcdot,\bigcdot}(\A,N)$ 
depends on $\S(\A)$ and on $\ker N$ (by relation \ref{rel_4}).

Let $N$ and $N'$ in $\operatorname{M}(n,l;\Q)$ be two defining matrices for two elliptic arrangements $\A$ and $\A'$ with the same poset of layers $\S=\S (\A)= \S(\A')$.
By \Cref{lemma:orbit_space}, there exist $g \in \operatorname{GL}(n;\Q)$ and $h \in \Z/2\Z \wr \mathfrak{S} _l$ such that
\[N'= g \cdot N \cdot h \]
Since $N$ and $N\cdot h$ define the same elliptic arrangement $\A$ in $E^n$, then
\[ A^{\bigcdot,\bigcdot}(\A,N) \simeq E^{\bigcdot,\bigcdot}(M(\A))\simeq A^{\bigcdot,\bigcdot}(\A,N \cdot h).\]
Notice that $\ker g \cdot N \cdot h= \ker N \cdot h$, so the two algebras $A^{\bigcdot,\bigcdot}(\A,N \cdot h)$ and $A^{\bigcdot,\bigcdot}(\A', g \cdot N \cdot h)$ are equal.

Let now $\A$ be an elliptic arrangement: then there exists an essential elliptic arrangement $\A'$ in $E^r$ such that
\[M(\A)\simeq M(\A') \times E^{n-r},\]
where $r$ is the rank of the poset $\S(\A)$.
Finally, we conclude that the cohomology algebra $H^\bigcdot(M(\A);\Q)$ depends only on $\S(\A)$ and $n$.
\end{proof}

\begin{example}
Let $\A$ be the arrangement in $E^2$ with divisors:
\begin{align*}
& P_1=O \\
& P_1 + 5 P_2 = O \\
& 2 P_1 + 5 P_2=O
\end{align*}
and $\A'$ defined by the divisors:
\begin{align*}
& P_1=O \\
& 2P_1 + 5 P_2 = O \\
& 3 P_1 + 5 P_2=O.
\end{align*}
An easy inspection shows that $\S(\A) \simeq \S(\A')$.
Defining matrix for $\A$ and $\A'$ are, for instance, $N=\left( \begin{smallmatrix}
1 & 1 & 2 \\
0 & 5 & 5
\end{smallmatrix} \right)$
and $N'=\left( \begin{smallmatrix}
2 & 1 & -3 \\
5 & 0 & -5
\end{smallmatrix} \right)$;
however it is more convenient to choose $N''=\left( \begin{smallmatrix}
1 & 2 & 3 \\
0 & 5 & 5
\end{smallmatrix} \right)= N' \cdot h$ as defining matrix for $A'$, where 
\[h=((1,1,-1), (1,2)) \in \Z/2\Z \wr \mathfrak{S}_3.\]
We use the multiplicative notation for $\Z/2\Z=\{1,-1\}$.

Notice that $\ker N \neq \ker N'$, but fortunately, $\ker N = \ker N''= \Q v$, where $v= \left( \begin{smallmatrix}
1 \\ 1 \\ -1
\end{smallmatrix} \right) \in \Q^3$, hence
\[ A^{\bigcdot,\bigcdot}(\A,N)= A^{\bigcdot,\bigcdot}(\A',N'' ) \simeq A^{\bigcdot,\bigcdot}(\A',N') .\]
Consider $G=\left( \begin{smallmatrix}
1 & \frac{1}{5} \\
0 & 1 
\end{smallmatrix} \right)$, then 
\[\left( \begin{matrix}
1 & 2 & 3 \\
0 & 5 & 5
\end{matrix} \right) = \left( \begin{matrix}
1 & \frac{1}{5} \\
0 & 1 
\end{matrix} \right)
\left( \begin{matrix}
1 & 1 & 2 \\
0 & 5 & 5
\end{matrix} \right).\]
Since $G$ does not belong to $\operatorname{GL}(2;\Z)$, the topological space $M(\A)$ and $M(\A')$ are not in general homeomorphic.
However, both of them are a Galois covering of the same elliptic arrangement, with Galois group $\Z/5\Z \times \Z/5\Z$.
\end{example}

\subsection{Vanishing of higher cohomology}

\begin{lemma}\label{lemma:affine_var}
Let $\A$ be an essential elliptic arrangement.
Than the complement $M(\A)$ is an affine variety.
\end{lemma}

Basic general facts about affine morphisms can be found in
\cite[\href{http://stacks.math.columbia.edu/tag/01S5}{Tag 01S5}]{stacks-project}.

\begin{proof}
Let $\A$ be an essential arrangement in $E^n$ and let $\A'$ be an essential sub-arrangement of $\A'$ with exactly $n$-divisors.
If $N\in \operatorname{M}(n,l;\Z)$ is a defining matrix for $\A$ of rank $n$, choose a square minor $N'$ of $N$ of rank $n$.
We can define $\A'$ as the sub-arrangement of $\A$ given by the divisors corresponding to the columns of $N'$.
If the arrangement $\A'$ is not unimodular, consider the subgroup $H$ of $E^n$ of translations that fix all divisors $D \in \A'$.
The quotient $E^n/H$ is isomorphic to $E^n$ and $M(\A')/H$ is the complement of a unimodular essential arrangement that we call $\A''$.
Obviously, $M(\A'')$ is isomorphic to $(E\setminus \{O\})^n$ and thus it is an affine variety.
Since $M(\A') \twoheadrightarrow M(\A'')$ is a Galois covering with Galois group $H$, the morphism is finite and therefore affine (by \cite[\href{http://stacks.math.columbia.edu/tag/01WN}{Tag 01WN}]{stacks-project}), hence $M(\A')$ is an affine variety.

For each $D_i \in \A$ consider the maps $\pi_i: E^n \rightarrow E$ as in \Cref{lemma:iso_dga} and their restrictions to $M(\A') \rightarrow E$.
If $D_i$ belongs to the arrangement $\A'$, then they restrict to $M(\A') \rightarrow E\setminus \{O\}$.
Call $\pi': M(\A')\rightarrow (E\setminus \{O\})^n \times E^{l-n}$ the map $(\pi_1, \dots, \pi_l)$ (suppose that the first $n$ terms are the divisors in $\A'$) and $\pi:M(\A)\rightarrow (E\setminus \{O\})^l$ the map $(\pi_1, \dots, \pi_l)$.
Then the following diagram, where the horizontal maps are the natural open embeddings, is cartesian.
\begin{center}
\begin{tikzpicture}
\node (xa1) at (0,0) {$(E\setminus \{O\})^l$};
\node (xa2) at (4,0) {$(E\setminus \{O\})^n \times E^{l-n}$};
\node (x1) at (0,2) {$ M(\A)$};
\node (x2) at (4,2) {$M(\A')$};
\draw[->] (xa1) to (xa2);
\draw[->] (x1) to node[left]{$\pi$} (xa1);
\draw[->] (x2) to node[right]{$\pi' $} (xa2);
\draw[->] (x1) to  (x2);
\end{tikzpicture}
\end{center}
Since $(E\setminus \{O\})^l$ is affine, the lower map is affine (by \cite[\href{http://stacks.math.columbia.edu/tag/01SI}{Tag 01SI}]{stacks-project}) and so is the top inclusion (by \cite[\href{http://stacks.math.columbia.edu/tag/01SD}{Tag 01SD}]{stacks-project}).
Finally, $M(\A)$ is affine because $M(\A')$ is affine and the inclusion are affine morphisms.
\end{proof}

\begin{thm} \label{thm:vanishing}
Let $\A$ be an elliptic arrangement in $E^n$ of rank $r$ (in particular, if $\A$ is essential then $r=n$).
The cohomology modules $H^i(M(\A);\Z)$ vanish when $i>2n-r$.
\end{thm}

\begin{proof}
As previously observed, there exists an essential arrangement $\A'$ such that $M(\A)\simeq M(\A')\times E^{n-r}$.
By K\" unnet theorem it is thus enough to prove that $H^i(M(\A'))$ vanishes for $i>r$.
By \Cref{lemma:affine_var}, $M(\A')$ is a complex affine variety of real dimension $2r$ and by Andreotti and Frankel theorem (\cite[Theorem 3.1.1]{Lazarsfeld04}) or by Artin and Grothendieck theorem (\cite[Theorem 3.1.12]{Lazarsfeld04}) we get that $H^i(M(\A'))=0$ for every $i>r$.
\end{proof}

As consequence of \Cref{thm:vanishing}, if $\A$ is an essential arrangement in $E^n$, then the first quadrant spectral sequence $E_2^{p,q}(M(\A))$ has non-zero entries only in the triangle $p+2q \leq 2n$.
However, the third page $E_3^{p,q}(M(\A))$ is non-zero only for $p+q\leq n$.

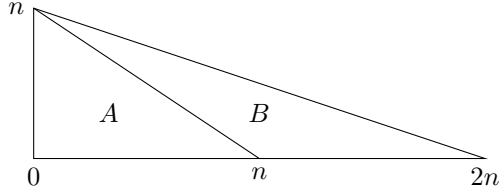
\begin{figure}
\centering
\begin{tikzpicture}
\node (a) at (1,0.6) {$A$};
\node (b) at (3,0.6) {$B$};
\draw (0,0) node[anchor=north]{$0$}
  -- (6,0) node[anchor=north]{$2n$}
  -- (0,2) node[anchor=east]{$n$}
  -- cycle;
\draw (0,2)
  -- (3,0) node[anchor=north]{$n$};
\end{tikzpicture}
\caption{The second page is supported on the triangle $A\cup B$. 
The third page only on triangle $A$.}
\end{figure}
\section{The braid elliptic arrangement}\label{sec:braid_arr}

We study the particular case of the braid elliptic arrangements.
In $E^n$ consider the braid arrangement $A_{n-1}$ defined by the divisors 
\[ D_{i,j}=\{P \in E^n \mid P_i = P_j \}\]
for all $1 \leq i<j\leq n$.
In this case $l=\binom{n}{2}$ and we can order the divisors by following the lexicographical order.

The space $M(A_{n-1})$ coincides with $C(n,E)$, the configuration space of $n$ distinct ordered points over an elliptic curve $E$.

The configuration space of $n$ points in a topological space $X$ has been deeply studied (e.g. \cite{FMacP94}, \cite{Kriz94}), unfortunately the cases of even dimensional compact varieties seem the harder ones.

In the case of the braid arrangement, the model $E_2^{\bigcdot,\bigcdot}(M(A_{n-1}))$ coincides with the Kriz model $E_\bigcdot^\bigcdot$ introduced in \cite{Kriz94} up to shifting the degrees, i.e.
\[ E^{p,q}_2(M(\A_{n-1})) \simeq E^{p+q}_q.\]

This elliptic arrangement is unimodular, but it is not essential, in fact every divisor contains the elliptic curve
\[C=\{P \in E^n \mid P_1=P_2 = \cdots =P_n\}.\]
The group $E$ acts on $E^n$ by translation:
\[Q\cdot (P_1, \dots, P_n) \stackrel{\textnormal{def}}{=} (P_1+Q, P_2+Q, \dots P_n+Q)\]
and this action preserves $M(A_{n-1})$, so that there exists a splitting 
\[ M(A_{n-1})=\faktor{M(A_{n-1})}{E} \times C ,\]
where the first factor is an essential arrangement in $E^{n-1}$.
As consequence, the dga $E_2^{\bigcdot,\bigcdot}(M(A_{n-1}))$ splits as
\[E_2^{\bigcdot,\bigcdot}(M(A_{n-1})) \simeq E_2^{\bigcdot,\bigcdot}(M(A_{n-1})/E) \otimes_\Q H^\bigcdot(E;\Q), \]
where $ H^\bigcdot(E;\Q)$ is a dga concentrated in degrees $(\bigcdot,0)$ with trivial differentials.

\subsection{The braid hyperplane arrangement}
Notice that, since the poset of layers $\S(A_{n-1})$ has a maximal element (i.e. the curve $C$), the first column $E^{0,\bigcdot}_2 (M(A_{n-1}))$ is an algebra isomorphic to the cohomology algebra of the complement of braid hyperplane arrangement.
A lot is known about it, the most important can be found in \cite{OT92} and \cite{DeCP11}.

The Orlik-Solomon algebra is isomorphic to the cohomology algebra with integer (or rational) coefficients and a basis is given by the decreasing forest on $n$-vertices.

A \textit{forest} is an acyclic graph on the set $\{1, \dots,n\}$ and any of its connected components is called a \textit{tree}.
The \textit{root} of a tree is its vertex with the greatest label.
A \textit{descending} tree is a tree in which the unique path between the root and a vertex has decreasing labels, while a decreasing forest is a forest made of decreasing trees.
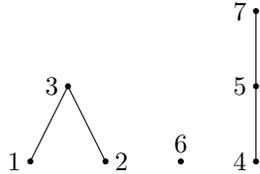
\begin{center}
\begin{tikzpicture}
\fill (0,0) circle (.04cm) node[left]{$1$};
\fill (1,0) circle (.04cm) node[right]{$2$};
\fill (0.5,1) circle (.04cm) node[left]{$3$};
\fill (2,0) circle (.04cm) node[above]{$6$};
\fill (3,0) circle (.04cm) node[left]{$4$};
\fill (3,1) circle (.04cm) node[left]{$5$};
\fill (3,2) circle (.04cm) node[left]{$7$};
\draw[thin] (0,0) -- (0.5,1) -- (1,0);
\draw[thin] (3,0) -- (3,1) -- (3,2);
\end{tikzpicture}
\end{center}

A basis of the cohomology algebra of the braid hyperplane arrangement is given by the cocycles $\omega_F$ where $F$ runs over all decreasing forests (as proven in \cite[Corollary~15]{CFV15}).
Any element $\omega_F$ has degree equal to the number of edges in $F$ and 
$\omega_F \omega_{F'}=0$ if the union graph $F\cup F'$ has a cycle.
In the case that $F\cup F'$ is a tree, the OS-relations allow us to write the product $\omega_F \omega_{F'}$ in terms of the generators in a recursive way.

The dimension of $H^k(M(A_{n-1}^H))$ is $\stirling{n}{n-k} $, the Stirling number of first kind.
The Poincaré polynomial of the hyperplane arrangement is 
\[P_{n-1}(t)=\prod_{j=1}^{n-1}(jt+1) \]
by \cite[Corollary 2]{Arnold69}, while the Tutte polynomial $T_n(x,y)$ was calculated in \cite[Theorem 5.1]{GS94} and has exponential generating function  
\begin{equation}
\sum_{n=1}^{\infty} \frac{T_n(x,y) t^n}{n!}= \frac{\left( \sum_{n=0}^{\infty} (\frac{t}{y-1})^n\frac{y^{\binom{n}{2}}}{n!} \right) ^{(y-1)(x-1)}-1 }{x-1}.
\end{equation}

The Tutte polynomial can be computed easily using the recursion formula in \cite{Pak}:
\[ T_{n+1}(x,y)= \sum_{k=1}^{n} \binom{n-1}{k-1}(x+y+y^2+ \dots + y^{k-1}) T_{k}(1,y)T_{n+1-k}(x,y).\]
For any hyperplane arrangement $\A^H$ the Poincaré polynomial is a specialization of the Tutte polynomial, e.g. for the braid arrangement we have:
\[P_{n-1}(x)=x^{n-1} T_n \left(\frac{1}{x},0 \right). \]

There is a natural action of $\mathfrak{S}_n$ on the braid arrangement $A^H_n$ given by permutation of coordinates.
Thus the cohomology algebra is a representation of the symmetric group.
This is described initially in top degree by \cite[Theorem 7.3]{Stanley82}
\begin{equation} \label{eq:repr_top}
H^{n-1}(M(A^H_{n-1})) = \sgn_n \cdot \Ind_{Z_n}^{\mathfrak{S}_n} \zeta_n,
\end{equation} 
where $Z_n$ is a subgroup of $\mathfrak{S_n}$ generated by an $n$-cycle and $\zeta_n$ is a non-trivial character of $Z_n$.
Later, in \cite[Theorem 4.5]{LS86}, a description for every degree has been given:
\[ H^k(M(A^H_n)) = \sgn_n \cdot \bigoplus_{\substack{\lambda \vdash n \\ |\lambda|=n-k}} \Ind_{Z(\lambda)}^{\mathfrak{S}_n} \xi_\lambda ;\]
this decomposition is a special case of \Cref{thm:weak_dec_E2} for $p=0$.

\subsection{The first two columns of the third page}

The following result is a straightforward consequence of \cite[Proposition 1.2]{AAB14}.
\begin{thm}\label{thm:AAB14}
For all $n>0$ the differentials $d_2^{0,q}$ are injective, therefore
\[ E^{0,q}_3(M(A_n))=0 \]
for all $q>0$.
\end{thm}
Moreover, $E^{0,0}_1(M(A_n))$ is a one dimensional $\Q$-vector space.
For each pair $(i,j)$, let $\omega_{i,j}$ be the generator of the cohomology group $H^1(M(A^H_{n-1})) \simeq E^{0,1}_2(M(A_{n-1}))$ corresponding to the divisor associated to the reflection $(i,j) \in \mathfrak{S}_n$.
A circuit $C=(c_1, \dots, c_k)$ of length $k$ (for $k>2$) is a simple cycle in the complete graph $K_n$ with $k$ edges, with a fixed orientation of the cycle, and a starting vertex.
For all edges $c$ in $C$ we denote with $s(c)$ the starting vertex and with $t(c)$ the vertex in which the edge terminates.
Hence, for a circuit $C$ we have $s(c_1)=t(c_k)$, $t(c_i) = s(c_{i+1})$ for all $i=1, \dots , k-1$ and $s(c_i)\neq s(c_j)$ for $i \neq j$.

\begin{example}
The cycles of length $k>2$ in the complete graph $K_n$ are $\frac{1}{2} \binom{n}{k} (k-1)!$ (equal to half the number of $k$-cycles in $\mathfrak{S}_n$).
The circuits of length $k>2$ in the complete graph $K_n$ are $\binom{n}{k} k!$.
\end{example}

\begin{defi}\label{def:omega_C}
For every ordered set of edges $C=(c_1,\dots, c_k)$ define the element $\omega_{C} \in E^{0,k}_2(M(A_n))$ as
\[ \omega_C \stackrel{\textnormal{def}}{=} \omega_{c_1} \omega_{c_2} \cdots \omega_{c_k},\]
where $\omega_c=\omega_{s(c),t(c)}\in E^{0,1}_2(M(A_{n-1}))$.
\end{defi}

Notice that if $C$ contains a circuit, then $\omega_{C}=0$, otherwise $C$ is a forest.
If $C$ is a forest, let $I(C)$ be the set
\[ I(C) \defeq \{(s(c),t(c)) \mid c \in C\}\]
and $W(C)$ the layer
\[ W(C) \defeq \{(P_1, \dots, P_n) \mid P_{s(c)}=P_{t(c)} \textnormal{ for all } c \in C\}.\]
The isomorphism $\varphi : A^{\bigcdot,\bigcdot} \rightarrow E_2^{\bigcdot,\bigcdot}(M(A_{n-1}))$ of \Cref{lemma:iso_dga} maps $\omega_{W(C),I(C)}$ to $\omega_C$.
Let $\pr_i: E^n \rightarrow E$ be the projection onto the $i^{\textnormal{th}}$ factor and consider $x_i= \pr_i^*(x)$ and $y_i=\pr_i^*(y)$.

\begin{defi}
For all circuits $C=(c_1, \dots, c_k)$ of the complete graph $K_n$, let $L_C, L'_C \in E_2^{1,k-2}(M(A_n))$ be the elements
\[ L_C \defeq \sum_{1\leq i <j \leq k} (-1)^{i+j} (x_{t(c_i)} - x_{s(c_i)}) \omega_{C\setminus \{c_i, c_j\} }\]
and
\[ L'_C \defeq \sum_{1\leq i <j \leq k} (-1)^{i+j} (y_{t(c_i)} - y_{s(c_i)}) \omega_{C\setminus \{c_i, c_j\} }\]
\end{defi}

Let $\tilde{C}$ be the circuit obtained from $C=(c_1, \dots, c_k)$ by changing the orientation, i.e. $\tilde{C}=(\tilde{c}_k, \tilde{c}_{k-1}, \dots, \tilde{c}_1)$ where $s(\tilde{c}_i)=t(c_i)$ and $t(\tilde{c}_i)=s(c_i)$.
The following relation holds:
\[ L_{\tilde{C}}=(-1)^{\binom{k-2}{2}}L_C,\]
indeed
\begin{align*}
L_{\tilde{C}} &= \sum_{1\leq i <j \leq k} (-1)^{i+j} (x_{s(c_{k-i})} - x_{t(c_{k-i})}) \omega_{\tilde{C} \setminus \{c_{k-i}, c_{k-j}\} } \\
&= \sum_{1\leq a < b \leq k} (-1)^{k-b+k-a} (x_{s(c_b)} - x_{t(c_a)}) \omega_{\tilde{C} \setminus \{c_a, c_b\} } \\
&= \sum_{1\leq a < b \leq k} (-1)^{a+b} (x_{t(c_a)} - x_{s(c_b)}) \omega_{\tilde{C} \setminus \{c_a, c_b\} } \\
&= (-1)^{\binom{k-2}{2}} \sum_{1\leq a < b \leq k} (-1)^{a+b} (x_{t(c_a)} - x_{s(c_b)}) \omega_{C \setminus \{c_a, c_b\}} \\
&= (-1)^{\binom{k-2}{2}} L_C.
\end{align*}

\begin{thm}
Let $C$ be a circuit in the complete graph $K_n$.
Then the elements $L_C$ and $L'_C$ belong to $\ker d_2^{1,k-2}$.
\end{thm}

\begin{proof}
We compute $\dd_2 (L_C)$, the calculation of $\dd_2 (L'_C)$ being completely analogous:
\begin{align*}
\dd_2 (L_C)&=- \sum_{1 \leq i < j \leq k} (-1)^{i+j}(x_{t(c_i)} - x_{s(c_i)}) \dd_2(\omega_{C\setminus \{c_i, c_j\}}) \\
&= - \sum_{1 \leq i < j \leq k} \sum_{a \neq i,j} (-1)^{i+j+ \sigma(i,j,a)} u_{i} u_{a} v_{a} \omega_{C\setminus \{c_i, c_j,c_a\})},
\end{align*}
where $u_{i}=x_{t(c_i)} - x_{s(c_i)}$, $v_{i}=y_{t(c_i)} - y_{s(c_i)}$ for all oriented edge $c_i$ and $\sigma(i,j,a)=a$ if $i<a<j$, $\sigma(i,j,a)=a+1$ otherwise.
Let $1 \leq \alpha < \beta < \gamma \leq k$ be a triple of integers and calculate the coefficient $n(\alpha,\beta,\gamma)$ of $\omega_{C \setminus \{ c_\alpha, c_\beta, c_\gamma\}}$ in $\dd_2(L_C)$.
Up to a sign, this coefficient is equal to
\begin{multline*}
 n(\alpha,\beta,\gamma)= (-1)^{\alpha+\beta+ \sigma(\alpha,\beta,\gamma)} u_\alpha u_\gamma v_\gamma + (-1)^{\alpha+\gamma+ \sigma(\alpha,\gamma,\beta)} u_\alpha u_\beta v_\beta +\\
 + (-1)^{\beta+\gamma+ \sigma(\beta,\gamma,\alpha)}  u_\beta u_\alpha v_\alpha .
\end{multline*}
Notice that $v_\gamma \omega_{C \setminus \{ c_\alpha, c_\beta, c_\gamma\}}=-(v_\alpha + v_\beta) \omega_{C \setminus \{ c_\alpha, c_\beta, c_\gamma\}}$,
hence we can rewrite the $n(\alpha,\beta,\gamma)$ as:
\begin{multline*}
 n(\alpha,\beta,\gamma)= - (-1)^{\alpha+\beta+ \sigma(\alpha,\beta,\gamma)} u_\alpha u_\gamma v_\beta + (-1)^{\alpha+\gamma+ \sigma(\alpha,\gamma,\beta)} u_\alpha u_\beta v_\beta +\\
 - (-1)^{\alpha+\beta+ \sigma(\alpha,\beta,\gamma)} u_\alpha u_\gamma v_\alpha + (-1)^{\beta+\gamma+ \sigma(\beta,\gamma,\alpha)}  u_\beta u_\alpha v_\alpha
\end{multline*}
Analogously, we have $u_\alpha u_\gamma \omega_{C \setminus \{ c_\alpha, c_\beta, c_\gamma\}}=-u_\alpha u_\beta \omega_{C \setminus \{ c_\alpha, c_\beta, c_\gamma\}}$, thus:
\begin{align*}
n(\alpha,\beta,\gamma) =& \,\begin{multlined}[t][10cm]
  ((-1)^{\alpha+\beta+ \sigma(\alpha,\beta,\gamma)}+ (-1)^{\alpha+\gamma+ \sigma(\alpha,\gamma,\beta)}) u_\alpha u_\beta v_\beta +\\
 +(-(-1)^{\alpha+\beta+ \sigma(\alpha,\beta,\gamma)} + (-1)^{\beta+\gamma+ \sigma(\beta,\gamma,\alpha)})  u_\beta u_\alpha v_\alpha
\end{multlined}\\
= & \begin{multlined}[t][10cm] ((-1)^{\alpha+\beta+\gamma+1}+ (-1)^{\alpha+\gamma+\beta}) u_\alpha u_\beta v_\beta +\\
+(-(-1)^{\alpha+\beta+\gamma} + (-1)^{\beta+\gamma+\alpha})  u_\beta u_\alpha v_\alpha
\end{multlined}\\
= & 0. \qedhere
\end{align*}
\end{proof}

\begin{remark}
The same elements $L_C$ and $L'_C$ can be defined for graphical arrangements and more generally for all unimodular circuits in an elliptic arrangement.
In the last case, the equations of the divisors must be chosen such that $D_i= \pi_i^{-1}(Q_i)$ for $i=1,\dots k$ and $\pi_1 + \cdots + \pi_{k-1}=\pi_k$.
Hence, the elements
\[ \sum_{1\leq i <j \leq k} (-1)^{i+j} \pi_i^*(x) \omega_{C\setminus \{c_i, c_j\} },\]
for every circuit $C$ and $x \in H^1(E)$, represent non-zero cohomology classes.
\end{remark}

\subsection{Labelled and weighted decreasing forests}
A basis for the vector space $E_2^{\bigcdot,\bigcdot}(M(A_{n-1}))$ is given by the labelled decreasing forests.

\begin{defi}
A \textit{labelled decreasing forest} $(F,s)$ is a decreasing forest $F$ with a labelling $s$ of the roots with values in $\{1,x,y,xy\}$ (the basis of $H^\bigcdot(E)$).
The bi-degree of $(F,s)$ is 
\[\deg (F,s)= (\sum_{i \in R} \deg s(i), |F|),\]
where $R$ is the set of the roots of $F$, $\deg: \{1,x,y,xy\} \rightarrow \{0,1,2\}$ is the degree function and $|F|$ is the number of edges in $F$. 
\end{defi}

Every forest $F$ in $K_n$ can be considered as a list of edges, putting the edges in lexicographical order.
As discussed in the previous subsections (see \Cref{def:omega_C}), consider -- for each forest $F$ -- the element $\omega_F\in E^{0,|F|}_2(M(A_{n-1}))$.
For each labelled decreasing forest $(F,s)$, define the element
\[ \omega_{F,s} \defeq \prod_{i\in R} \pi_i^*(s(i)) \omega_F \in E^{\deg (F,s)}_2 (M(A))\]
where the product is in increasing order.


\begin{figure}
\centering
\begin{tikzpicture}
\fill (0,0) circle (.04cm) node[left]{$1$};
\fill (1,0) circle (.04cm) node[right]{$2$};
\fill (0.5,1) circle (.04cm) node[label={[blue]45:$x$}, left]{$3$};
\fill (2,0) circle (.04cm) node[label={[blue]45:$1$}, left]{$6$};
\fill (3,0) circle (.04cm) node[left]{$4$};
\fill (3,1) circle (.04cm) node[left]{$5$};
\fill (3,2) circle (.04cm) node[label={[blue]45:$xy$}, left]{$7$};
\draw[thin] (0,0) -- (0.5,1) -- (1,0);
\draw[thin] (3,0) -- (3,1) -- (3,2);
\end{tikzpicture}
\caption{A labelled forest, in blue the labels. The associated element is $x_3x_7y_7\omega_{1,3}\omega_{2,3}\omega_{4,5}\omega_{5,7}$.} \label{fig:F_s}
\end{figure}
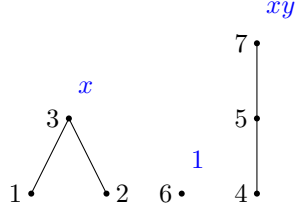

\begin{prop}
The set $\{\omega_{F,s}\}$, where $(F,s)$ runs over all labelled decreasing forests, is a basis of $E^{\bigcdot,\bigcdot}_2(M(A_n))$ as $\Q$-vector space.
\end{prop}

\begin{defi}
A \textit{weighted decreasing forest} is a decreasing forest $F$ with a function $t: R \rightarrow \{ 0,1,2\}$.
The bi-degree of $(F,t)$ is 
\[\deg (F,s)= \left(\sum_{i\in R} t(i), |F| \right).\]
\end{defi}

Each labelled forest $(F,s)$ defines the weighted forest $(F,\deg \circ s)$.

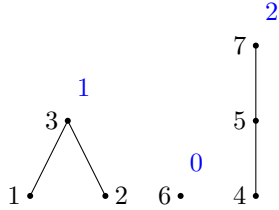
\begin{figure}
\centering
\begin{tikzpicture}
\fill (0,0) circle (.04cm) node[left]{$1$};
\fill (1,0) circle (.04cm) node[right]{$2$};
\fill (0.5,1) circle (.04cm) node[label={[blue]45:$1$}, left]{$3$};
\fill (2,0) circle (.04cm) node[label={[blue]45:$0$}, left]{$6$};
\fill (3,0) circle (.04cm) node[left]{$4$};
\fill (3,1) circle (.04cm) node[left]{$5$};
\fill (3,2) circle (.04cm) node[label={[blue]45:$2$}, left]{$7$};
\draw[thin] (0,0) -- (0.5,1) -- (1,0);
\draw[thin] (3,0) -- (3,1) -- (3,2);
\end{tikzpicture}
\caption{The weighted forest $(F,t)$ associated to the forest of \Cref{fig:F_s}, in blue the labels.}\label{fig:F_t}
\end{figure}

\begin{defi}
The \textit{support} of a forest $F$ is the partition $\supp(F)$ of $[n]$ whose blocks are the sets of vertices in the same connected component.
The \textit{shape} of $F$ is $\shape (F)$, the partition of the number $n$ associated with $\supp(F)$ (i.e. the size of the trees in $F$).

If $(F,s)$ (resp. $(F,t)$) is a labelled (resp. weighted) decreasing forest, we can label (resp. weigh) the blocks of $\shape (F)$ and the ones of $\supp(F)$ with the function $s$ (resp. $t$).
\end{defi}

The support of $F$ is the partition defined in \cite[section 3.3 and remark 3.8]{Bibby17} associated with the layer on which $\omega_{F,s}$ is supported (this does not depend on $s$).

\begin{example} \label{example:supp_shape}
Let $(F,s)$ and $(F,t)$ be the labelled and weighted forests as in \Cref{fig:F_s,fig:F_t}.
The support of $F$ is $\supp(F)=\{\{1,2,3\},\{6\}, \{4,5,7\}\}$ with labelling and weight:
\[\begin{tabular}{c|c|c|c}
 & \{1,2,3\} & \{6\} & \{4,5,7\} \\ 
\hline 
$s$ & $x$ & $1$ & $xy$ \\ 
\hline 
$t$ & $1$ & $0$ & $2$ \\ 
\end{tabular} 
\]
The shape of the labelled forest $(F,s)$ is the partition represented by the tableau 
\[
\begin{Young}
    $xy$ & & \cr
    $x$  & & \cr
    $1$ \cr
\end{Young}
\]
where every row is labelled; we put the labels in the first column.
The shape of the weighted forest $(F,t)$ is
\[
\begin{Young}
      $2$ & & \cr
      $1$ & & \cr
      $0$\cr
\end{Young}.
\]
\end{example}

Let $(F,s)$ be a labelled forest (now we do not assume the forest is decreasing) and let $t= \deg \circ s$ be the weight.
The support of $(F,s)$ is $\Sigma \vdash [n]$ and whose shape is $\lambda \vdash n$.
Define the following subspaces of $E^{\deg(F,s)}_2(M(A_{n-1}))$:
\begin{align*}
& E(F,t)= \Q \langle \omega_{F,s'} \mid t= \deg \circ s' \rangle \\
& E(\Sigma,s)= \Q \langle \omega_{F',s} \mid \supp (F',s)=(\Sigma,s) \rangle \\
& E(\Sigma,t)= \bigoplus_{\deg \circ s'=t} E(\Sigma,s') \\
& E(\lambda,s)= \bigoplus_{\Sigma'} E(\Sigma',s) \\
& E(\lambda,t)= \bigoplus_{\Sigma'} E(\Sigma',t).
\end{align*}

Notice that $E^{p,q}_2(M(A_{n-1}))= \bigoplus_{\lambda,t} E(\lambda,t)$, where the sum is over all partitions $\lambda \vdash n$ into $n-q$ blocks and all labels $t$ of $\lambda$ such that $\sum_{i=1}^{n-q} t(i)=p$.

\begin{remark}\label{remark:sum_repr}
The vector space $E(\lambda,s)$ is stable under the action of $\mathfrak{S}_n$ and $E(\lambda,t)$ is stable under the action of $\mathfrak{S}_n \times \SL_2(\Q)$.
\end{remark}

\subsection{Linear independence of \texorpdfstring{$L_C$}{LC}}
We call a circuit $C$ \textit{standard} if the two greatest vertices in $C$ are adjacent and we denote the dimension of $E_2^{p,q}(M(A_n))$ (respectively, of $E_3^{p,q}(M(A_n))$) with 
$e_2^{p,q}(n)$ (resp. $e_3^{p,q}(n)$).

\begin{thm}\label{thm:independency}
The elements $L_C, L'_C$, where $C$ runs over all standard circuits, are linearly independent, so $e_3^{1,q}(n) \geq 2\binom{n}{q+2} q!$.
\end{thm}

\begin{question}\label{claim:E3_1q}
Does the equality $e_3^{1,q}(n)=2\binom{n}{q+2} q!$ hold?
\end{question}

A \textit{forest of bamboo} is a graph on $[n]$ such that every vertex has degree at most two; a \textit{standard forest of bamboo} is a forest of bamboo such that the maximum vertex of every connected component is an extremity of the bamboo (i.e., a vertex of degree one).
The element $\omega_B$ (recall \Cref{def:omega_C}) is associate to each standard forest of bamboos, considered as a list of its edges in lexicographical order.


\begin{lemma}\label{lemma:basis_bamboo}
The set
\[ \{\omega_B \mid B \textnormal{ is a standard forest of bamboo}\}\]
is a basis of $H^\bigcdot (M(A^H_{n-1}))$.
\end{lemma}

\begin{example}
The forest in \Cref{fig:F_s} (forgetting the labelling) is a forest of bamboo, but not a standard forest of bamboo.
However, the monomial $\omega_{1,3}\omega_{2,3} \omega_{4,5} \omega_{5,7}$ is equal to 
\[ \omega_{1,2}\omega_{2,3} \omega_{4,5} \omega_{5,7}
- \omega_{1,2} \omega_{1,3} \omega_{4,5} \omega_{5,7}
\]
in $H^4 (M(A^H_6))$ that corresponds to two standard forest of bamboo.

\begin{figure}
\centering
\begin{subfigure}[b]{0.49\textwidth}
\centering
\begin{tikzpicture}
\fill (1,0) circle (.04cm) node[left]{$1$};
\fill (1,1) circle (.04cm) node[left]{$2$};
\fill (1,2) circle (.04cm) node[left]{$3$};
\fill (2,0) circle (.04cm) node[left]{$6$};
\fill (3,0) circle (.04cm) node[left]{$4$};
\fill (3,1) circle (.04cm) node[left]{$5$};
\fill (3,2) circle (.04cm) node[left]{$7$};
\draw[thin] (1,0) -- (1,1) -- (1,2);
\draw[thin] (3,0) -- (3,1) -- (3,2);
\end{tikzpicture}
\caption{A standard forest of bamboo whose associated element is $\omega_{1,2}\omega_{2,3} \omega_{4,5}\omega_{5,7}$.} 
\end{subfigure}
\begin{subfigure}[b]{0.49\textwidth}
\centering
\begin{tikzpicture}
\fill (1,0) circle (.04cm) node[left]{$2$};
\fill (1,1) circle (.04cm) node[left]{$1$};
\fill (1,2) circle (.04cm) node[left]{$3$};
\fill (2,0) circle (.04cm) node[left]{$6$};
\fill (3,0) circle (.04cm) node[left]{$4$};
\fill (3,1) circle (.04cm) node[left]{$5$};
\fill (3,2) circle (.04cm) node[left]{$7$};
\draw[thin] (1,0) -- (1,1) -- (1,2);
\draw[thin] (3,0) -- (3,1) -- (3,2);
\end{tikzpicture}
\caption{A standard forest of bamboo whose associated element is $\omega_{1,2}\omega_{1,3}\omega_{4,5}\omega_{5,7}$.} 
\end{subfigure}
\caption{two standards forest of bamboo.}
\label{fig:forest_bamboo}
\end{figure}
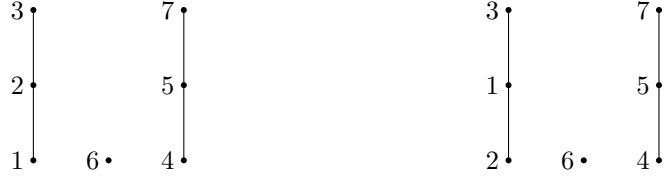
\end{example}

\begin{proof}[Proof of \Cref{thm:independency}]
We prove the theorem by contradiction: let 
\[ \sum_{C \textnormal{ standard}} a_C L_C = \sum_{C \textnormal{ standard}} a'_C L'_C \]
be a non-trivial linear combination.
Suppose there exists a standard circuit $\tilde{C}$
such that $a_{\tilde{C}} \neq 0$.
The case in which $a'_{C} \neq 0$ is analogous.
Let $\tilde{B}$ be the standard bamboo obtained by deleting from $\tilde{C}$ the two edges with end the maximal vertex.
Since $\tilde{B}$ is a forest, we can label all roots with $1$ except the maximal vertex of $\tilde{C}$ which we label with $x$.
Let $(\Sigma,s)$ be the support of $\tilde{B}$; consider the projection $p:E^{1,q}_2(M(A_{n-1})) \rightarrow E(\Sigma,s)$ and notice that $E(\Sigma,s)\simeq H^{q}(M(A^H_{q}))$.
We have $p(L_{\tilde{C}})=\omega_{\tilde{B}}$ and $p(L'_C)=0$ for all standard circuits $C$.
Hence, we have
\[ \sum_{C \textnormal{ standard}} a_C p(L_C) = 0.\]
If $C$ is a circuit with at least one vertex not contained in the circuit $\tilde{C}$, then $p(L_C)=0$, otherwise $p(L_C)=\omega_B$ where $B$ is the standard bamboo obtained by deleting from $C$ the two edges with end the maximal vertex.
By \Cref{lemma:basis_bamboo}, the elements $p(L_C)$ are linearly independent, so $a_{\tilde{C}}=0$ and we obtain a contradiction.
\end{proof}


\section{Representation theory}\label{sec:rep_theory}
\subsection{The action of the symmetric group}

Recall \Cref{eq:repr_top}, i.e. the description of $E^{0,n-1}_2(M(A_{n-1}))$ as representation of $\mathfrak{S}_n$.
The following formula describes the multiplicity of an irreducible representation $V_\lambda$ of $\mathfrak{S}_n$, for any $\lambda \vdash n$.
\begin{align*}
 \langle \sgn_n \Ind_{C_n}^{\mathfrak{S}_n} \zeta_n, V_\lambda \rangle_{\mathfrak{S}_n} &=
 \langle \zeta_n , \Res_{C_n}^{\mathfrak{S}_n} \sgn_n V_\lambda \rangle_{C_n} \\
 &= \frac{1}{n}\sum_{k|n} \mu \left(\frac{n}{k} \right) (-1)^{(n-1)k} \chi_{V_\lambda}(\sigma^k),
\end{align*}
where $\mu$ is the M\" obius function, $\chi$ the character of the representation and $\sigma$ any $n$-cycle in $\mathfrak{S}_n$.

More generally, the $\mathfrak{S}_n$-representation $E^{p,q}_2(M(A_{n-1}))$ was described in \cite[Theorem 3.15]{AAB14}.
As observed in \Cref{remark:sum_repr}, there is a decomposition
\[ E^{p,q}_2(M(A_{n-1})) = \bigoplus_{\lambda,s} E(\lambda,s),\]
where the sum is taken over distinct labelled partitions of $n$.


For each labelled partition $(\lambda, s)$, we choose a permutation $\sigma \in \mathfrak{S}_n$ of shape $\lambda$ and let $H(\lambda)$ be the subgroup of $\mathfrak{S}_n$ generated by the cycles of $\sigma$.
The abelian subgroup $H(\lambda)$ depends on the choice of $\sigma$ but any two of such groups are conjugated.
Now, we label arbitrarily the cycles of $\sigma$ with the function $s$, compatibly with the labels of $\lambda$.
Let $N(\lambda,s)$ be the subgroup of the centralizer of $\sigma$ generated by the permutations that exchange the cycles of $\sigma$ with the same labels (e.g. if $\sigma_i$ and $\sigma_j$ are two cycles of $\sigma$ such that $|\sigma_i|=|\sigma_j|$, $s(i)=s(j)$, and $\sigma_i=g\sigma_j$, then $g \in N(\lambda,s)$).
Obviously, $N(\lambda,s)$ is a product of symmetric groups of smaller size.
We define $Z(\lambda,s) < \mathfrak{S}_n$ as the semi-direct product $H(\lambda) \rtimes N(\lambda,s)$.

\begin{example}\label{example:label_part}
Let $(\lambda,s)$ be the following labelled partition:
\[\begin{Young}
      $xy$ & \cr
      $xy$ & \cr
      $y$\cr
      $x$\cr
      $x$\cr
\end{Young}.\]
Consider the element 
\[\sigma= \underbrace{(1,2)}_{xy} \underbrace{(3,4)}_{xy} \underbrace{(5)}_{y} \underbrace{(6)}_{x}\underbrace{(7)}_{x}\]
in $\mathfrak{S}_n$.
The subgroup $H(\lambda)\simeq \Z/2\Z \times \Z/2\Z$ is generated by $(1,2)$ and $(3,4)$, the subgroup $N(\lambda,s)\simeq \mathfrak{S}_2 \times \mathfrak{S}_2$ is generated by $(1,3)(2,4)$ and $(6,7)$, and finally $Z(\lambda,s)$ is a group isomorphic to $(\Z/2\Z \wr \mathfrak{S_2}) \times \mathfrak{S_2}$.
\end{example}

Let $\zeta_n$ be a non-trivial character of the cyclic group and $\zeta_\lambda$ the character of $H(\lambda)\simeq \Z/\lambda_1\Z \times \cdots \times \Z/\lambda_r \Z$ given by
\[ \zeta_\lambda \defeq \zeta_{\lambda_1} \boxtimes \dots \boxtimes \zeta_{\lambda_r}.\]
Let $\alpha(\lambda,s)$ be the one dimensional representation of $N(\lambda,s)\simeq \mathfrak{S}_{\mu_1} \times \cdots \times \mathfrak{S}_{\mu_l}$ defined by 
\[ \alpha(\lambda,s) \defeq \sgn_{\mu_1}^{\otimes n_1} \boxtimes \cdots \boxtimes  \sgn_{\mu_l}^{\otimes n_l},\]
where $n_i=\lambda_i+\deg s(i) +1$.
Finally, define $\xi(\lambda,s)$ as the one dimensional representation $\zeta_\lambda \boxtimes \alpha(\lambda,s)$ of $Z(\lambda,s)$.

\begin{thm}[{\cite[Theorem~3.15]{AAB14}} or {\cite[Theorem~6.3]{DPR14}}]\label{thm:weak_dec_E2}
The following equalities of $\mathfrak{S}_n$ representations hold:
\begin{gather*}
 E(\lambda,s) = \sgn_n \otimes \Ind _{Z(\lambda,s)} ^{\mathfrak{S}_n} \xi (\lambda,s),\\
 E^{p,q}_2(M(A)) = \bigoplus_{\lambda,s} E(\lambda,s),
 \end{gather*}
where the sum is taken over all labelled partition $\lambda\vdash n$ with $n-q$ blocks and $\sum_i \deg(s(i))=p$.
\end{thm}

\begin{example}
Consider the labelled partition $(\lambda,s)$ of \Cref{example:label_part}, then
\[E(\lambda,s)= \sgn_7 \otimes \Ind_{Z(\lambda,s)}^{\mathfrak{S_7}} \xi, \]
where the characters are as shown in the following table:
\[ \begin{tabular}{c|c|c|c|c}
 & $(1,2)$ & $(3,4)$ & $(1,3)(2,4)$ & $(6,7)$ \\ 
\hline 
$\zeta$ & $-1$ & $-1$ &  &  \\ 
\hline 
$\alpha$ &  &  & $-1$ & $-1$ \\ 
\hline 
$\xi$ & $-1$ & $-1$ & $-1$ & $-1$ \\ 
\end{tabular}\]
\end{example}

\subsection{The action of \texorpdfstring{$\SL_2(\Q)$}{SL2(Q)}} \label{subsect:action_SL2}
Every homeomorphism of the elliptic curve into itself induces a pullback homomorphism in cohomology, so that the group of automorphisms of the curve acts on $H^\bigcdot(E;\Z)$.
This action factors through the group $\SL_2(\Z)$ and extends to the group $\SL_2(\Q)$ on the cohomology with rational coefficients.
We denote the irreducible representation of $\SL_2(\Q)$ of dimension $k+1$ by $\V_k$.
Thus $\SL_2(\Q)$ acts on $E^{p,q}_2(M(A_n))$ and this action is compatible with the one of $\mathfrak{S}_n$.
We have the decomposition:
\[E_2^{\bigcdot,\bigcdot}(M(A_{n-1})) \simeq E_2^{\bigcdot,\bigcdot}(M(A_{n-1})/E) \otimes (\mathds{1}_n \boxtimes \V_0 \oplus \mathds{1}_n \boxtimes \V_1 u \oplus \mathds{1}_n \boxtimes \V_0 u^2)\]
where $u$ is of bi-degree $(1,0)$ and $\mathds{1}_n$ is the trivial representation of $\mathfrak{S}_n$.

Consider $T\subset \SL_2(\Q)$, a maximal torus such that $t\cdot x= tx$ and $t\cdot y= t^{-1}y$ (under the identification of $T$ with $\Q^*$).
The irreducible representations of $T$ are $V(a)$ (one dimensional), $a \in \Z$.
The subspace $E(\lambda,s)$ is
\[ E(\lambda,s) = \sgn_n \otimes \Ind _{Z(\lambda,s)} ^{\mathfrak{S}_n} \xi (\lambda,s) \boxtimes V(a),\]
where the equality is intended as $\mathfrak{S}_n\times T$-representations and $a=|\{i \mid s(i)=x\}|-|\{i \mid s(i)=y\}|$.
Thus, for a fixed labelled partition $(\lambda,t)$, define \[E(\lambda,t,a) \defeq \bigoplus_s E(\lambda,s),\]
where $s$ runs over all labels such that $a=|\{i \mid s(i)=x\}|-|\{i \mid s(i)=y\}|$ and $\deg \circ s = t$.
Recall that the representation $\V_k$ of $\SL_2(\Q)$, once restricted to a representation of $T$, splits as
\begin{equation} \label{eq:spiegazione_per_C}
\V_k= \bigoplus_{a=0}^k V(2a-k).
\end{equation}
Since $E(\lambda,t)= \oplus_a E(\lambda,t,a) \boxtimes V(a)$, formula \eqref{eq:spiegazione_per_C} implies that
\begin{equation} \label{eq:dec_E2_SL2}
E(\lambda,t) = \bigoplus_{a=0}^p \left( E(\lambda,t,a)\ominus E(\lambda,t,a+2) \right) \boxtimes \V_a
\end{equation}
as $\mathfrak{S}_n\times \SL_2(\Q)$-representations (where $p= \sum_{i\in R} t(i)$).
\begin{remark} \label{remark:obvious_repr}
Notice that $E(\lambda,t,a)=0$ if $a \not \equiv \sum_i t(i) \; (\textnormal{mod } 2)$.
We use the standard notation $V_\mu$, $\mu \vdash n$ for the irreducible representation of $\mathfrak{S}_n$.
Each module $E^{p,q}_2(M(A_n))$, that appears in the second page of the Leray spectral sequence, decomposes in the following way as $\mathfrak{S}_n\times \SL_2(\Q)$-representations, for some $n(\mu,k) \in \N$:
\[E^{p,q}_2(M(A_n))= \bigoplus_{\mu,k} (V_\mu \boxtimes \V_k)^{\oplus n(\mu,k)}\]
where the sum is over all partitions $\mu \vdash n$, $k\equiv p \; (\textnormal{mod }2)$ and $0\leq k \leq p$.
Moreover, from eq.~\eqref{eq:dec_E2_SL2}, the virtual representation $E(\lambda,t,a)\ominus E(\lambda,t,a+2)$ is an actual representation since it coincides with the isotypical component $\V_a$ of the actual representation $E(\lambda,t)$.
\end{remark}

The discussion above can be resumed in the following theorem.
\begin{thm}\label{thm:strong_dec_E2}
The representation $E^{p,q}_2(M(A_n))$ of $\mathfrak{S}_n\times \SL_2(\Q)$ is
\[E^{p,q}_2(M(A_n)) = \bigoplus_{\substack{0 \leq a \leq p \\ |\lambda|=n-q \\ \sum_i t(i)=p }} \left( E(\lambda,t,a)\ominus E(\lambda,t,a+2) \right) \boxtimes \V_a,\]
where 
\[ E(\lambda,t,a) = \sgn_n \otimes \bigoplus_{\substack{\deg \circ s= t \\ |s^{-1}(x)|-|s^{-1}(y)|=a}}   \Ind _{Z(\lambda,s)} ^{\mathfrak{S}_n} \xi (\lambda,s).\]
Thus, the representation $\V_a$ occurs only in the small triangle of \Cref{fig:repr_Va}.
\begin{figure}
\centering
\begin{tikzpicture}
\draw (0,4) node[left]{$n-1$} -- (8,0) node[below]{$2n-2$} -- (0,0) node[align=left,below]{$(0,0)$} -- cycle;
\draw[pattern=north east lines, pattern color=blue] (2,2) node[left,above]{$(a,n-1-a)$} -- (6,0) node[below]{$2n-2-a$} -- (2,0) node[below]{$a$} -- cycle;
\draw [color=red] (0,4) -- (4,0);
\end{tikzpicture}
\caption{The representation $\V_a$ appears only in the darken triangle.} \label{fig:repr_Va}
\end{figure}
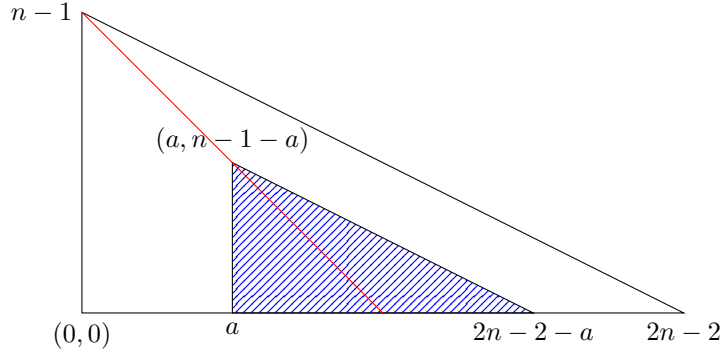
\end{thm}

\begin{defi}\label{def:e_k}
We denote by $e_2^{p,q}(n)_k$ (respectively $e_3^{p,q}(n)_k$) the multiplicity of $\V_k$ in $E^{p,q}_2(M(A_n))$ (resp. $E^{p,q}_3(M(A_n))$).
The multiplicity of $\V_k$ in $E^{p,q}_i(M(A_n)/E)$ is denoted by $e_i^{p,q}(n,E)_k$ for $i=2,3$.
\end{defi}

\subsection{The first row}
Now we analyse the first row $E^{\bigcdot,0}_2(M(A_n)/E)$ of the second page of the Leray spectral sequence both as graded algebra and as representation of $\mathfrak{S}_n \times \SL_2(\Q)$; from this we deduce some information about the first row $E^{\bigcdot,0}_3(M(A_n)/E)$ of the third page of the spectral sequence.

Let $V(1)_n$ be the standard representation of the symmetric group $\mathfrak{S}_n$.
We know that $E_2^{\bigcdot,0}(M(A_n)/E)$ and $ \bigwedge ^{\bigcdot} V(1)_n \boxtimes \V_1$ are isomorphic both as algebras and as representations.
A nice formula holds:
\begin{equation}\label{eq:Schur_dec}
\bigwedge^p V \boxtimes W = \bigoplus_{\lambda \vdash p} \mathbb{S}_\lambda V \boxtimes \mathbb{S}_{\lambda'} W.
\end{equation}
A proof can be found in \cite[Corollary 2.3.3]{Weyman03} or \cite[Exercise 6.11(b)]{FH91}.
In the formula above $\lambda'$ is the conjugate partition of $\lambda$ and $\mathbb{S}_\lambda$ is the Schur functor, see \cite{FH91} for a general reference.
Moreover, the dimension of $\mathbb{S}_\lambda V$ is
\[\dim \mathbb{S}_\lambda V= s_\lambda (\underbrace{1,\cdots,1}_{\dim V})= \prod_{1\leq i < j \leq \dim V} \frac{\lambda_i- \lambda_j +i-j}{i-j}, \]
where $s_\lambda$ is the Schur polynomial as proven in \cite[Theorem 6.3]{FH91}.
The representation $\mathbb{S}_{\lambda} \V_1$ is non-zero if and only if $\lambda=(a,b)$ and in this case we have the equality:
\[\mathbb{S}_{\lambda} \V_1 = \V_{a-b} .\]
\begin{example}
The isotypical component $\V_k$ of $E^{k,0}_2(M(A_n)/E)$ is $\mathbb{S}_\lambda V(1) \boxtimes \mathbb{S}_{\lambda'} W$ where $\lambda'=(k)$, therefore 
\[\mathbb{S}_\lambda V(1) \boxtimes \V_k =\bigwedge ^k V(1) \boxtimes \V_k = V(\underbrace{1,\dots,1}_{k}) \boxtimes \V_k, \]
where we adopt the common notation defined in \cite[Section 4.1]{Bibby17}.
\end{example}

\begin{lemma}\label{lemma:e2_pq(n)_k}
The numbers $e_2^{p,q}(n)_k$ and $e_2^{p,q}(n,E)_k$ can be calculated with the formulae
\begin{align*}
&e_2^{p,q}(n)_k = \stirling{n}{n-q}s_{(\frac{p+k}{2},\frac{p-k}{2})'} (\underbrace{1,\cdots,1}_{n-q}),\\ 
&e_2^{p,q}(n,E)_k = \stirling{n}{n-q}s_{(\frac{p+k}{2},\frac{p-k}{2})'} (\underbrace{1,\cdots,1}_{n-q-1}).
\end{align*}

\end{lemma}

\begin{proof}
The modules $E^{p,q}_2(n)$ and $E^{p,0}_2(n-q) ^{\oplus e_2^{0,q}(n)}$ are isomorphic as $\SL_2(\Q)$-representations, hence the 
component $\V_k$ appears in the module $E_2^{p,q}(M(A_{n-1}))$ with multiplicity $e_2^{0,q}(n) e_2^{p,0}(n-q)_k$ by \Cref{def:e_k}. 
Analogously, $\V_k$ appears in the module $E_2^{p,q}(M(A_{n-1})/E)$ with multiplicity $e_2^{0,q}(n) e_2^{p,0}(n-q,E)_k$.
The number $e_2^{0,q}(n)$ is the dimension of $E_2^{0,q}(M(\A_{n-1}))$, which is isomorphic to $H^q(M(\A_{n-1}))$, and is equal to the Stirling number $\stirling{n}{n-q}$.
The decomposition of \Cref{eq:Schur_dec} applied to $W=\V_1$ and $V=V(0)_{n-q} \oplus V(1)_{n-q}$ shows that $\bigwedge^p H^1(E^{n-q})= \bigwedge^p V \boxtimes W$ and that 
\[e_2^{p,0}(n-q)_k= s_{(\frac{p+k}{2},\frac{p-k}{2})'} (\underbrace{1,\cdots,1}_{n-q}).\]
Analogously, $e_2^{p,0}(n-q,E)$ is the dimension of the module $\bigwedge^p H^1(E^{n-q}/E)\simeq \bigwedge^p H^1(E^{n-q-1})$.
The first cohomology group $H^1(E^{n-q-1})$ is the tensor product of $ V(1)_{n-q-1}$ and $\V_1$, thereby 
\[ e_2^{p,0}(n-q,E)_k = s_{(\frac{p+k}{2},\frac{p-k}{2})'} (\underbrace{1,\cdots,1}_{n-q-1}). \qedhere \]
\end{proof}

Since $e_2^{k-2,1}(n)_k=0$ (\Cref{remark:obvious_repr}), $\bigwedge^k V(1) \boxtimes \V_k$ occurs in the third page $E^{k,0}_3(M(A_n)/E)$.
The range of $\dd^{p,1}_2$ is an ideal of $E_2^{\bigcdot,0}(M(A_{n-1})/E)$ generated in degree $2$.
Let $u_i=x_i-x_{i+1}$ and $v_i=y_i-y_{i+1}$ for $i=1,\dots,n-1$ be a basis of $E^{1,0}_2(M(A_n)/E)$.
The elements $u_iv_i$ are in the range of the differential, moreover the formula
\[\dd \omega_{i,j}= \left(\sum_{k=i}^{j-1} u_i\right) \left( \sum _{k=i}^{j-1} v_i\right)\]
shows that $u_iv_j+u_jv_i$ is in the range of the differential.
A basis for the vector space $E_2^{p,0}(M(A_{n-1})/E)/\im \dd$ is given by the monomials
\[ u_{i_1}\cdots u_{i_k}v_{i_{k+1}} \cdots v_{i_p}\]
for $1 \leq i_1 < \dots <i_p \leq n-1$ and $0\leq k \leq p$, so $E^{p,0}_3(M(A_{n-1})/E)$ is of dimension $\binom{n-1}{p} (p+1)$, which is equal to the dimension of $\bigwedge^p V(1) \boxtimes \V_p$.
Finally we obtain the following theorem.
\begin{thm}\label{thm:first_row}
The $p^{\textnormal{th}}$ term in the first row of the third page of the Leray spectral sequence is
\[E^{p,0}_3(M(A_{n-1})/E)= \bigwedge^p V(1) \boxtimes \V_p=V(\underbrace{1,\cdots,1}_{p})\boxtimes \V_p\]
and is of dimension $\binom{n-1}{p} (p+1)$.
\end{thm}
Hence we have found that $e_2^{p,0}(n)=e_2^{p,0}(n)_p= \binom{n-1}{p}$.

\subsection{Some lower bounds for Hodge numbers}

We calculate $e_3^{p,1}(n)_p$ in order to obtain a weak lower bound for the second row $E_3^{\bigcdot,1}(M(A_{n-1})/E)$ of the third page of the Leray spectral sequence.
\begin{prop}\label{prop:second_row}
The number $e_3^{p,1}(n,E)_p$ is equal to
\[e_3^{p,1}(n,E)_p= \binom{n}{p+2}\binom{p+1}{2}.\]
\end{prop}

\begin{proof}
By \Cref{thm:first_row} we have $e_3^{p+2,0}(n,E)_p=0$ and by \Cref{thm:strong_dec_E2} we have $e_2^{p-2,2}(n,E)_p=0$.
Thus the equality
\[e_3^{p,1}(n,E)_p= e_2^{p,1}(n,E)_p - e_2^{p+2,0}(n,E)_p\]
holds.
The first addendum of the right hand side is
\begin{align*}
e_2^{p,1}(n,E)_p &= e_2^{0,1}(n,E) e_2^{p,0}(n-1,E)_p = \stirling{n}{n-1} \binom{n-2}{p}   \\
&= \binom{n}{2}\binom{n-2}{p}
\end{align*}
by \Cref{lemma:e2_pq(n)_k} and the second one is
\begin{align*}
e_2^{p+2,0}(n,E)_p &= \dim \mathbb{S}_{(p+1,1)'} V(1)_{n}\\
&=  s_{(p+1,1)'}(\underbrace{1,1 \dots, 1}_{n-1})= \frac{n(n-1)}{p+2}\binom{n-2}{p}.
\end{align*}
A straightforward calculation proves the proposition:
\begin{align*}
e_3^{p,1}(n,E)_p &= \binom{n}{2}\binom{n-2}{p}-\frac{n(n-1)}{p+2}\binom{n-2}{p} \\
&= \binom{n}{2} \binom{n-2}{p}\left( 1-\frac{2}{p+2} \right) \\
&= \binom{n}{2} \binom{n-2}{p} \frac{p}{p+2}= \binom{n}{p+2} \binom{p+1}{2}. \qedhere
\end{align*}
\end{proof} 
We give a lower bound for $e_3^{k,n-1-k}(n,E)$ by calculating explicitly $e_3^{k,n-1-k}(n)_k$.

\begin{thm}\label{thm:top_left_E3k}
For every $p$ and $q$ the followings holds
\[ e_3^{p,q}(p+q+1,E)_p= \stirling{ p+q}{p}. \]
Equivalently, for $n\geq k+1$, the formula
\[ e_3^{k,n-1-k}(n,E)_k = \stirling{ n-1 }{ k} \]
holds.
\end{thm}

\begin{proof}
We prove the second statement.
\Cref{thm:vanishing} shows that the following long sequence is exact:
\[ E_2^{k,n-k-1} \rightarrow E_2^{k+2,n-k-2} \rightarrow \cdots \rightarrow E_2^{2n-k,0} \rightarrow 0.\]
The isotypical component $\V_k$ appears in the kernel of the first map with the following multiplicity:
\begin{align*}
e_3^{k,n-1-k}(n,E)_k & = \sum_{i=0}^{n-1-k}(-1)^i e_2^{k+2i,n-1-k-i}(n)_k \\
&= \sum_{i=0}^{n-1-k}(-1)^i \stirling{n}{k+i+1} s_{(k+i,i)'} (\underbrace{1,1 \dots, 1}_{k+i})\\
&= \sum_{i=0}^{n-1-k}(-1)^i \stirling{n}{k+i+1} \binom{k+i}{k}\\
&= \sum_{j=k}^{n-1}(-1)^{j+k} \stirling{n}{j+1} \binom{j}{k}.
\end{align*}
Finally, the formula $\stirling{n}{m}=\sum_{k=m}^n (-1)^{m-k}\stirling{n+1}{k+1} \binom{k}{m} $ of \cite[Formula 6.18 of Table 265]{Concrete_Mathematics} completes the proof.
\end{proof}

\begin{question}
Does the equality $e_3^{k,q}(n,E)_k=\stirling{p+q}{p} \binom{n}{p+q+1}$ hold?
\end{question}
\Cref{thm:top_left_E3k} gives a positive answer to the above question in the case $n=p+q+1$. Another consequence of \Cref{thm:top_left_E3k} for $k=p=1$ is a positive answer to \Cref{claim:E3_1q} in the particular case $n=q+2$.

\begin{remark}\label{rem:iso_rep_res}
The map 
\[p:E^{1,n-2}_2(M(A_{n-1})) \rightarrow E(\Sigma,s)\simeq H^{n-2}(M(A^H_{n-2}))\]
of \Cref{thm:independency} does not depend on the circuit $\tilde{C}$ and induces a map 
\[E^{1,n-2}_3(M(A_{n-1})) \rightarrow H^{n-2}(M(A^H_{n-2}))\otimes \V_1.\]
This map is an isomorphism of $\mathfrak{S}_{n-1} \times \SL_2(\Q)$-representation, where the action of the group $\mathfrak{S}_{n-1}$ on $E^{1,n-2}_2(M(A_{n-1}))$ is the restriction of the natural action of $\mathfrak{S}_{n}$.
\end{remark}

Recall the action of $\mathfrak{S}_n$ on $H^{n-2}(M(A^H_{n-1}))$ defined in \cite{Gaiffi96}.

\begin{question}
Is the map $E^{1,n-2}_2(M(A_{n-1})) \rightarrow H^{n-2}(M(A^H_{n-2}))\boxtimes \V_1$ an isomorphism of $\mathfrak{S}_n$-modules?
\end{question}

\section{\texorpdfstring{$1$}{1}-formality for graphic arrangements} \label{sec:formality}
The configuration spaces of points on an elliptic curve are not always formal, as shown in \cite{Bezrukavnikov} and \cite{DPS}.
In \cite[Theorem 4.16]{BibbyPhD} it is proven that a chordal graphic arrangement (with at least one cycle) has a non-formal complement.

We will show that the complement of a graphic arrangement is $1$-formal if and only if it has no cycles of length $3$.
This result is a particular case of \cite[Theorem~1.2]{BMPP2017}, but the proof is completely different.

\subsection{Formality result}
Let $G=(\mathcal{V}, \mathcal{E})$ be a finite graph without loops and multiple edges.

\begin{defi}
A \textit{graphic elliptic arrangement} is the arrangement in $E^{|\mathcal{V}|}$ given by the divisors
\[ D_{i,j} = \{ P \in E^{|\mathcal{V}|} \mid P_i=P_j\}\]
for each edge $(i,j) \in \mathcal{E}$.
\end{defi}

\begin{example}
The braid arrangement $A_n$ is the graphic arrangement associated with the complete graph over $n$ vertexes.
\end{example}

Let $(E,d)$ and $(E',d')$ be two commutative differential graded algebras (briefly, CDGA).
\begin{defi}
A $q$\textit{-isomorphism} $\varphi:E \rightarrow E'$ is a morphism of algebras and of cochain complexes such that the induced morphism in cohomology $\varphi^*: H^k(E) \rightarrow H^k(E')$ is an isomorphism for $k\leq q$ and an injection for $k=q+1$.

The CDGA $E$ and $E'$ are $q$\textit{-equivalent} if there exists a zig-zag of $q$-isomorphisms connecting $E$ to $E'$.
Moreover, the CDGA $E$ is $q$\textit{-formal} if $E$ is $q$-equivalent to its cohomology $(H^\bigcdot(E),0)$.

A finite type CW-complex $X$ is $q$\textit{-formal} if its Sullivan model (defined in \cite{Sullivan}) is $q$-formal.
\end{defi}

Notice that if $X$ is of finite dimension $n$, then formality and $n$-formality coincide.
Moreover, the Sullivan model is $q$-equivalent to the second page of the Leray filtration for all $q$, so from now on we will work on the CDGA $E_2^{\bigcdot,\bigcdot}(M(\A))$ (endowed with the total degree) and discuss the formality of $M(\A)$.

Let $G$ be a graph over $n$ vertexes and $M(G)$ the complement of the graphic arrangement defined by $G$.
Consider the inclusion $M(A_n) \hookrightarrow M(G)$ and the corresponding injection in the second pages of Leray spectral sequences
\[ E_2^{\bigcdot,\bigcdot}(M(G)) \hookrightarrow E_2^{\bigcdot,\bigcdot}(M(A_n)).\]
This map induces an injection of $E_3^{p,q}(M(G)) \hookrightarrow E_3^{p,q}(M(A_n))$ only when $p=0,1$.

\begin{lemma}\label{lemma:vanishing_graph}
Let $G$ be a graph without cycles of length $3$.
Then the vector spaces $E_3^{0,1}(M(G))$, $E_3^{0,2}(M(G))$ and $E_3^{1,1}(M(G))$ are zero.
\end{lemma}

\begin{proof}
The first two vector spaces vanish for any graphic elliptic arrangement.
We are left to prove that $E_3^{1,1}(M(G))=0$, or equivalently that
\[ E_2^{1,1}(M(G)) \bigcap \ker \dd_2^{1,1} =0 \quad \textnormal{in } E_2^{1,1}(M(A_n)).\]
By \Cref{thm:independency} and \Cref{prop:second_row}, a basis of the kernel $\ker \dd_2^{1,1}$ is given by $\{L_C,L_C'\}_C$ where $C$ runs over all (standard) circuits of length $3$.
Suppose that the element
\[z \defeq \sum_C a_C L_C + a'_C L'_C \in \ker \dd_2^{1,1},\] where $a_C, a'_C \in \Q$, belongs to $E_2^{1,1}(M(G))$.

For each cycle $C=\{(i,j),(j,k),(k,i)\}$ there is an edge of $C$ (suppose without loosing of generality $(i,j)$) does not lie in $G$.
Thus, the monomial $x_k \omega_{i,j}$ in $z$ has coefficient $\pm a_C$ and it must be zero since $z$ belongs to $E_2^{1,1}(M(G))$.
An analogous result holds for $a'_C$ and by arbitrariness of $C$ we conclude that $z=0$.
So $E_2^{1,1}(M(G)) \bigcap \ker \dd_2^{1,1} =0$ and \Cref{lemma:vanishing_graph} follows.
\end{proof}

\begin{lemma}\label{lemma:no_3_1_formal}
Let $G$ be a graph without cycles of length $3$.
Then the topological space $M(G)$ is $1$-formal.
\end{lemma}

\begin{proof}
It is enough to prove that $E_2^{\bigcdot,\bigcdot}(M(G))$ is $1$-formal. Consider the ideal $I$ of $E_2^{\bigcdot,\bigcdot}(M(G))$ generated by $E_2^{p,q}(M(G))$ for $q>0$ and $\im \dd_2$.
This ideal is stable with respect to $\dd_2$, so that the algebra $Q := E_2^{\bigcdot,\bigcdot}(M(G))/I$ is a CDGA with trivial differential.
Moreover, the projection $p: E_2^{\bigcdot,\bigcdot}(M(G)) \rightarrow Q$ is a morphism of CDGA and by \Cref{lemma:vanishing_graph} the map $p$ is a $1$-isomorphism.

Consider now the map
\[\varphi: E_3^{\bigcdot,\bigcdot}(M(G)) = \faktor{\ker \dd}{\im \dd} \hookrightarrow \faktor{E_2^{\bigcdot,\bigcdot}(M(G))}{\im \dd} \twoheadrightarrow Q,\]
it is a $1$-isomorphism and therefore $(E_2^{\bigcdot,\bigcdot}(M(G)),\dd)$ and $(E_3^{\bigcdot,\bigcdot}(M(G)),0)$ are $1$-equivalent.
\end{proof}

Let $(E,\dd)$ be a CDGA; for each $x \in H^1(E,\dd)$ define a new differential $\dd_x: E \rightarrow E$ by
\[\dd_x(e)= xe + \dd (e). \]
An easy check shows that $(E,\dd_x)$ is a CDGA.

\begin{defi}
The first resonance variety of a CDGA $(E,\dd)$ is
\[ \R^1(E) \defeq \{ x \in H^1(E,\dd) \mid H^1(E,\dd_x)\neq 0\}.\]
\end{defi}
Notice that if $\dd=0$, $\dd_x$ is the left multiplication by $x$.

\begin{lemma}\label{lemma:1_formal_no_3}
Let $G$ be a graph with a cycle of length $3$, then the resonance variety $\R^1(E_2^{\bigcdot,\bigcdot}(M(G)))$ is strictly contained in $\R^1(E_3^{\bigcdot,\bigcdot}(M(G)))$.
\end{lemma}

\begin{proof}
By \Cref{thm:AAB14} we have $H^1(E_2^{\bigcdot,\bigcdot}(M(G)), \dd) = E_2^{1,0}(M(G))=H^1(E^n)$.
Let $z=\sum_{i=1}^n a_ix_i+ b_i y_i$ be an arbitrary element of $H^1(E^n)$.
For each edge $(j,k)\in \mathcal{E}$ we have
\[ \dd_z(\omega_{j,k})= z\omega_{j,k} + (x_j-x_k)(y_j-y_k).\]
Since $\dd_z: E^{1,0}_2 \bigoplus E^{0,1}_2 \rightarrow E^{2,0}_2 \bigoplus E^{1,1}_2$ is upper triangular (i.e. preserves the Leray filtration) we consider the graded map $\gr \dd_z$.
Its kernel is $z\Q$ if $z \neq a(x_j-x_k)+b(y_j-y_k)$ for $a,b \in \Q$ and $(i,j) \in \mathcal{E}$.
Otherwise, the equations $\dd_z(b\omega_{j,k}-y_j+y_k)=0$ and $\dd_z(a\omega_{j,k}+x_j-x_k)=0$ ensure that $H^1(E_2^{\bigcdot,\bigcdot}(M(G)), \dd_z)$ is non-zero.
Therefore, $\R^1(E_2^{\bigcdot,\bigcdot}(M(G)))$ coincide with
\[ \bigcup_{(j,k) \in \mathcal{E}} \langle x_j-x_k, y_j-y_k \rangle .\]

Let $\{(i,j),(j,k),(k,i)\}$ be a circuit of the graph $G$, the image of $\dd_2^{0,1}$ contains 
\begin{align*}
& (x_i-x_j)(y_i-y_j),\\
& (x_i-x_k)(y_i-y_k),\\
& (x_j-x_k)(y_j-y_k).
\end{align*}
For sake of notation we call $u_1=x_i-x_j$, $u_2=x_i-x_k$, $v_1=y_i-y_j$ and $v_2=y_i-y_k$.
In $E_3^{2,0}(M(G))$ the relations $u_1v_1=0$, $u_2v_2=0$ and $u_1v_2+u_2v_1=0$ hold and hence 
\[(anu_1+amu_2+bnv_1+bmv_2)(nv_1+mv_2) =0 \]
in $E_3^{2,0}(M(G))$, for $a,b,n,m \in \Q$.
The above equation implies that the resonance variety $\R^1(E_3^{\bigcdot,\bigcdot}(M(G)))$ contains a $3$-dimensional quadric and it so contains strictly the $2$-dimensional (singular) variety $\R^1(E_2^{\bigcdot,\bigcdot}(M(G)))$.
\end{proof}

\begin{thm}
A graphic elliptic arrangement $M(G)$ is $1$-formal if and only if the graph $G$ does not contain cycles of length $3$.
\end{thm}

\begin{proof}
One implication is proven in \Cref{lemma:no_3_1_formal}, the other follows from \Cref{lemma:1_formal_no_3} together with the tangent cone theorem \cite[Theorem 14]{Suciu16}.
\end{proof}

We have discuss deeply the $1$-formality for graphic elliptic arrangements, but the following questions remain open.
Let $\A$ be an elliptic arrangement.
\begin{question}
When $\R^i(E^{\bigcdot,\bigcdot}_2(M(\A)))$ and $\R^i(E^{\bigcdot,\bigcdot}_3(M(\A)))$ coincides?
\end{question}
\begin{question}
When $M(\A)$ is formal?
\end{question}

As consequence of \Cref{thm:comb_depend} these resonance variety and the property of being $k$-formal, for $k \in \N$, are combinatorial determine, i.e. they depend only on the poset of layers.

\section*{Acknowledgements}
The author would like to thank his advisor Filippo Callegaro, Jacopo (Jack) D'Aurizio, and Angela Veronese for the useful discussions and for the review work.

\printbibliography
\end{document}